\documentclass[12pt,a4paper,final]{amsart}

\usepackage{version}
\usepackage[notcite,notref]{showkeys}
\usepackage{amssymb}
\usepackage{color}
\usepackage{cite}

\newtheorem{theorem}{Theorem}[section]
\newtheorem{lemma}[theorem]{Lemma}
\newtheorem{corollary}[theorem]{Corollary}

\theoremstyle{definition}

\newtheorem{example}[theorem]{Example}

\theoremstyle{remark}
\newtheorem{remark}[theorem]{Remark}

\numberwithin{equation}{section}

\DeclareMathOperator{\diam}{diam}
\DeclareMathOperator{\dimh}{dim_H}
\DeclareMathOperator{\dimhk}{dim_H^\kappa}
\DeclareMathOperator{\dimp}{dim_P}
\DeclareMathOperator{\dimpk}{dim_P^\kappa}

\DeclareMathOperator{\dimbu}{\overline{dim}_B}

\DeclareMathOperator{\Id}{Id}

\def\h{{\bf h}}
\def\x{{\bf x}}
\def\w{{\bf w}}
\def\z{{\bf z}}
\def\L{{\mathcal L}}
\def\N{{\mathbb N}}
\def\P{{\mathbb P}}
\def\E{{\mathbb E}}
\def\R{{\mathbb R}}
\def\T{{\mathbb T}}

\begin{document}

\title[Hitting probabilities of random covering sets in tori]
{Hitting probabilities of random covering sets in tori and metric spaces}

\author[E. J\"arvenp\"a\"a et al.]{Esa J\"arvenp\"a\"a$^1$}
\address{Mathematics, P.O. Box 3000,
         90014 University of Oulu, Finland$^{1,2,5}$}
\email{esa.jarvenpaa@oulu.fi$^1$}

\author[]{Maarit J\"arvenp\"a\"a$^2$}
\email{maarit.jarvenpaa@oulu.fi$^2$}

\author[]{Henna Koivusalo$^3$}
\address{Department on Mathematics, University of York, York YO10 5DD,
Great Britain$^3$}
\email{henna.koivusalo@york.ac.uk$^3$}

\author[]{Bing Li$^4$}
\address{Department of Mathematics, South China University of Technology,
Guangzhou, 510641, P. R. China$^4$}
\email{scbingli@scut.edu.cn$^4$}

\author[]{Ville Suomala$^5$}
\email{ville.suomala@oulu.fi$^5$}

\author[]{Yimin Xiao$^6$}
\address{Department of Statistics and Probability, A-413 Wells Hall,
Michigan State University, East Lansing MI48824, USA$^6$}
\email{xiao@stt.msu.edu$^6$}

\thanks{We acknowledge the support of the Academy of Finland, the Centre of
Excellence in Analysis and Dynamics Research. HK thanks the support of
EPSRC Grant EP/L001462 and Osk. Huttunen foundation. BL was supported by
NSFC 11201155, Guangdong Natural Science Foundation 2014A030313230 and
"Fundamental Research Funds for the Central Universities" SCUT (2015ZZ055).
YX was supported in part by NSF grants DMS-1307470 and DMS-1309856.}

\date{\today}

\begin{abstract}
We provide sharp lower and upper bounds for the Hausdorff dimension of the
intersection of a typical random covering set with a fixed analytic set both in
Ahlfors regular metric spaces and in the $d$-dimensional torus. In metric
spaces, we consider covering sets generated by balls and, in
tori, we deal
with general analytic generating sets.
\end{abstract}

\maketitle

\section{Introduction}

\subsection{Background}
Given a set $X$ and a sequence of subsets $Z_n\subset X$, $n\in\N$, a general
covering problem asks when $X$ is covered or when $X$ is covered infinitely
often by the sets $Z_n$. If $X$ is not covered, it is natural to 
study the size
and structure of the uncovered set $X\setminus\bigcup_{n\in\N} Z_n$ as well as
the limsup set $\limsup_{n\to\infty} Z_n=\bigcap_{k\in\N}\bigcup_{n=k}^\infty Z_n$.
This kind of problems arise in many different fields of mathematics. For
example, the generating sets $Z_n$ can come from a predescribed, arbitrary
sequence as in \cite{BeresnevichVelani06,WWX15}.  Given more 
structure to the
sequence,  finer information on the limsup set can be obtained, 
as demonstrated for instance in the situations where the sets are defined 
dynamically
\cite{FLL10,FanSchmelingTroubetzkoy,LiaoSeuret}, through shrinking targets
\cite{HV95} or in relation to continued fractions \cite{DF01,LWWX14}.

A classical problem of this sort, and perhaps a great motivator in the study of
limsup sets, is the problem of Diophantine approximation. In Diophantine
approximation, the object under study is the set of points that can be
approximated with a given approximation speed
$\psi\colon\mathbb N\to\mathbb R_+$ by rationals, that is, the set
\begin{equation}\label{eq:well}
\mathcal W(\psi)=\{x\in \mathbb R\mid |x-\tfrac pq|\le \psi(q)
  \text{ for infinitely many coprime pairs }
(p,q)\in\mathbb{Z}\times\mathbb N\}.\
\end{equation}
This is a limsup set of shrinking balls centred at rational points, see
\cite{H98,S80}. The question of which irrational points are `close' to rational
points is equally interesting in higher ambient dimensions, and in that case
there is a certain amount of freedom in choosing the shapes of the generating
sets, leading to  various kinds of interesting 
problems. For example, approximation by
cubes corresponds to simultaneous approximation at the same speed in all
directions, aligned rectangles correspond to different approximation speeds
and multiplicative approximation leads to approximating sets given by
hyperbolic regions \cite{BPV11,DV97,R98}. Lately, there has been an
increasing amount of interest in the Diophantine approximation properties of
points in fixed subsets of $\mathbb R^d$, for example, manifolds \cite{B12} or
fractal sets \cite{EFS11,W01}.

A further approach, and the one we will concentrate on, is to let the sets
$Z_n$ be random. In this context, the limsup set
$E=\bigcap_{k\in\N}\bigcup_{n=k}^\infty Z_n$ is usually referred to as a
{\it random covering set}.  The study of random covering sets
has a long and convoluted history; we refer the interested reader to
\cite{Dvoretzky56, Fan95, Fan02, Kahane85, Kahane92,
Kahane2000, Janson86, Shepp72, Shepp72-2} for the wide variety of 
interesting problems.

 In the present paper,  we study the {\it hitting
probabilities} of a random covering set $E$, namely, 
the probability that $E$ intersects a given subset $F\subset X$.
For some classical random sets, such as Brownian paths and 
fractal percolation
limit sets, the study of the hitting probabilities and the size of the
intersections goes back to Dvoretzky, Erd\"os, Kakutani and Taylor
\cite{DEK,DEKT}, Hawkes \cite{Hawkes1971b,Hawkes1981} and Lyons \cite{Lyons},
see also \cite{Kahane85,Peres96}. When the randomness appears as a random
transformation of a given subset of the Euclidean space, such results originate
from the pioneering works of Kahane \cite{Kahane86} and Mattila
\cite{Mattila84}, see also \cite{M95}. A recent line of research concerns
replacing the fixed set $F$ by a suitable parameterised family $\Gamma$ of sets
and showing that $\P(E\cap F\neq\emptyset\text{ for all }F\in\Gamma)>0$. See
\cite{ShmerkinSuomala2014} and references therein.

The structures of the random covering sets and more general limsup random
fractals are different from those considered in the above references. For
instance, for random covering sets, the packing and Hausdorff dimensions are
typically different. The hitting probabilities of a family of discrete random
limsup sets were studied by Khoshnevisan, Peres and Xiao in \cite{KPX}, with
applications to fast points of Brownian motion and other stochastic processes.
For further results concerning dimensional properties of limsup random
fractals, see \cite{DeMa,DPRZ,OrTa,Zh12}.

Applying the method of \cite{KPX}, Li, Shieh and Xiao \cite{LiShiehXiao}
investigated the intersection properties of random covering sets in the circle
$\T^1$. It is proved in \cite{LiShiehXiao} that if $(r_n)_{n\in\N}$ is a given
sequence of positive numbers and $\x=(x_n)_{n\in\N}$ is a sequence of independent
uniformly distributed random variables on the circle $\T^1$, then, denoting the
closed ball with radius $r$ and centre $x$ by $B(x,r)$, the random covering set
\[
E(\x)=\limsup_{n\to\infty}B(x_n,r_n)
\]
avoids a given analytic set $F\subset\T^1$ for almost all sequences
$\x=(x_n)_{n\in\N}$ provided  that $\dimp F<1-\dimh E(\x)$, while
$E(\x)\cap F\neq\emptyset$ for almost all $\x$ if $\dimp F>1-\dimh E(\x)$ and
if $(r_n)_{n\in\N}$ satisfies a mild technical assumption. Here,
$\dimh E(\x)=\inf\{s\le 1\mid\sum_{n=1}^\infty(r_n)^s<\infty\}$
for almost all $\x$. In the latter case,
\cite{LiShiehXiao} contains estimates for the Hausdorff dimension of
$E(\x)\cap F$. 

Several authors have considered random covering problems in
the context of metric spaces (see \cite{Kahane2000} and the references therein),
but their emphasis has been on the case of full covering and on the size of the
uncovered set $X\setminus\bigcup_{n\in\N}Z_n$. To the best of our knowledge, the
intersection properties of the random covering sets have been studied only
for randomly placed balls on  tori \cite{Du,Du2,LiShiehXiao,LS}. 
This has motivated us to investigate hitting probabilities of 
random covering sets on more general metric spaces.

\subsection{Overview of results and methods}

This work contains two types of results on hitting probabilities of random
covering sets. On one hand, we will consider the case when $X$ is a compact
Ahlfors regular metric space and $Z_n=B(x_n,r_n)$, where $(r_n)_{n\in\N}$ 
is a deterministic sequence of positive real numbers and $(x_n)_{n\in\N}$ 
is a sequence of independent random variables distributed according to an 
Ahlfors regular probability measure $\mu$ on $X$. On the other hand, we 
will consider the random covering sets
\begin{equation}\label{torusset}
\limsup_{n\to\infty} (x_n+A_n)
\end{equation}
in the $d$-dimensional torus $\T^d$, where $(A_n)_{n\in\N}$ is a fixed sequence
of analytic sets and $(x_n)_{n\in\N}$ are independent and uniformly distributed
random variables in $\T^d$. Finally, in $\T^d$ we will also consider a model
where the deterministic sets $A_n$ are, in addition to translating, also
randomly rotated.

Denoting the underlying space ($\T^d$ or a more general metric space) by $X$
and, for all $\x\in X^\N=:\Omega$, the random covering set by $E(\x)$,
we will be interested in the probability that the random covering set
intersects a fixed analytic set $F\subset X$. That is, we want to study the
hitting probability
\[
\P(\{\x\in\Omega\mid E(\x)\cap F\neq\emptyset\}),
\]
and, if this is positive, we shall determine almost sure
upper and lower bounds for $\dimh(E(\x)\cap F)$.

In Section~\ref{metric}, we will discuss the results of Li, Shieh and Xiao
\cite{LiShiehXiao} in detail, and extend them from the Euclidean
setting to Ahlfors regular metric spaces (see Theorems \ref{hittingcovering} and
\ref{intersectionestimate}). These theorems give sharp lower and upper bounds
for the Hausdorff dimension of the intersection of a typical random covering
set with a fixed analytic set $F$ in terms of the Hausdorff and packing
dimensions of $F$ and the covering sets. We remark that, in \cite{LiShiehXiao},
Li, Shieh and Xiao derived their results from hitting probability estimates
for the discrete limsup random fractals obtained in \cite{KPX} and
their proof could also be extended to metric spaces. However, in
Section~\ref{metric}, we will give new proofs which avoid the use of the
discrete limsup fractals. In Section \ref{parallel}, we consider a class of
carpet type affine covering sets to demonstrate how our results apply in the
Euclidean setting also when the results of \cite{LiShiehXiao} do not.

For the methods of the proof in Section~\ref{metric}, as well as in
\cite{LiShiehXiao}, it is essential that the generating sets are balls or
ball-like. In Section~\ref{torus}, we consider the case of analytic generating
sets $(A_n)_{n\in\N}$ in the $d$-dimensional torus and the random covering sets
defined in \eqref{torusset}. With methods completely different from those in
Section~\ref{metric},
applying classical intersection results of Mattila \cite{Mattila84}, we prove
in Section \ref{torus} that the upper bounds given in Theorems
\ref{hittingcovering} and \ref{intersectionestimate} are valid also in this
general setting (see Theorem~\ref{upper}). However, the counterparts of the
lower bounds are not true in this generality as shown by
Example~\ref{bad case}. To overcome the problem presented in
Example~\ref{bad case}, at the end of Section~\ref{torus}, we give a
modification of the model where the generating sets are randomly rotated. In
this modified model, under a classical extra assumption on the dimensions,
a sharp almost sure lower bound for the Hausdorff dimension of the intersection
of a random covering set with a fixed analytic set is discovered. For full
details, see Theorem~\ref{hittingresult_dimh}.

Intersection properties of random sets very similar to ours have been under
consideration in the context of Diophantine approximation, see Bugeaud and
Durand \cite{BD}. They use them to support a conjecture \cite[Conjecture 1]{BD}
related to the irrationality exponent of points in the middle third
Cantor set. We discuss this connection in detail in Remarks~\ref{rem:BD} and
\ref{conclution}.

\section{Random covering sets in metric spaces}\label{metric}

\subsection{Notations and definitions}

In this section, we consider random covering sets in the context of $t$-regular
metric spaces. Assume that $(X,\rho)$ is a compact metric space endowed with a
Borel probability measure $\mu$ which is Ahlfors $t$-regular for some constant
$0<t<\infty$. Recall that $\mu$ is Ahlfors $t$-regular if
there exists a constant $0<C<\infty$ such that
\begin{equation}\label{eq:d-regular}
C^{-1} r^t\le \mu(B(x,r))\le C r^t
\end{equation}
for all $x\in X$ and $0<r<\diam X$, where  $B(x,r)$ 
 is the closed ball in metric $\rho$
centred at $x$ with radius $r$ and $\diam X$ is the diameter of $X$.
Given $A\subset X$ and $\varepsilon>0$, an {\it $\varepsilon$-net of $A$} is a
subset $Y\subset A$ such that
$A\subset\bigcup_{y\in Y}B(y,\varepsilon)$ and $\rho(x,y)\ge\varepsilon$ for all
$x,y\in Y$ with $x\neq y$. We denote Hausdorff, packing and upper box
counting dimensions in the metric space $(X,\rho)$ by $\dimh$, 
$\dimp$ and $\dimbu$, respectively. In a couple of places we have two different
metrics in the same space. Then we put the metric as a superscript to
emphasise which metric is used to calculate the dimension.

In $X$, we will need an analogue of dyadic cubes, which we define next. For all
integers $n\ge1$, assume that $\mathcal{Q}_n=\{Q_{n,i}\}_{i\in I_n}$ is a finite
family of pairwise disjoint Borel subsets of $X$ such that
$\bigcup_{i\in I_n} Q_{n,i}=X$ and, moreover,
for each $Q_{n,i}\in\mathcal{Q}_n$, there is $x_{n,i}\in X$ such that
\begin{equation}\label{inoutsideballs}
B(x_{n,i},2^{-n})\subset Q_{n,i}\subset B(x_{n,i}, C 2^{-n}),
\end{equation}
where $C>0$ is a universal constant. Further, we assume that the families
$\mathcal{Q}_n$ are nested: for $i\neq j$ and $m\ge n$, either
$Q_{m,j}\subset Q_{n,i}$ or $Q_{m,j}\cap Q_{n,i}=\emptyset$. For convenience, we
also define $\mathcal{Q}_0=\{X\}$. Set
$\mathcal{Q}=\bigcup_{n=0}^\infty\mathcal{Q}_n$. We recall that, starting from a
nested family of $(2^{-n})$-nets, such finite families $\mathcal{Q}_n$ may be
constructed in any metric space which satisfies a mild doubling condition, in
particular, in any $t$-regular metric space, see \cite{KRS}.

Next we define random covering sets. Let $(\Omega,\mathcal A,\P)$ be the
completion of the infinite product of $(X,\mathcal B(X),\mu)$, where
$\mathcal B(X)$ is the Borel $\sigma$-algebra on $X$. Let $(r_n)_{n\in\N}$ be a
decreasing sequence of positive numbers tending to zero.
For all $\x\in\Omega$, the covering set is defined as
\[
E(\x)=\limsup_{n\to\infty}B(x_n,r_n)
  =\bigcap_{k=1}^\infty\bigcup_{n=k}^\infty B(x_n,r_n).
\]

\subsection{Dimension and hitting probabilities of random covering sets}
\label{sec:hitting_metric}

It is easy to see that $\mu(E(\x))=0$ for all $\x\in\Omega$ if
$\sum_{n=1}^\infty (r_n)^t<\infty$, whereas by the Borel-Cantelli lemma and
Fubini's theorem, $\mu(E(\x))=1$ for $\P$-almost all
$\x\in\Omega$ if $\sum_{n=1}^\infty(r_n)^t=\infty$. There is a concrete
formula for the almost sure Hausdorff dimension of the random covering set:

\begin{theorem}\label{dimcoveringset}
For $\P$-almost all $\x\in\Omega$,
\[
\dimh E(\x)=\inf\biggl\{s\le t\mid\sum_{n=1}^\infty (r_n)^s<\infty\biggr\}
 =\limsup_{n\to\infty}\frac{\log n}{-\log r_n},
\]
with the convention $\inf\emptyset=t$.
\end{theorem}

This result is well known if $X$ is the $d$-dimensional torus $\T^d$, and
similar techniques can be used to extend the result to $t$-regular metric
spaces. For instance, applying the mass transference principle
\cite[Theorem 3]{BeresnevichVelani06} by Beresnevich and Velani, the almost
sure value of $\dimh E(\x)$ can be determined exactly as in
\cite[Proposition 4.7]{JJKLS}. We note that Theorem~\ref{dimcoveringset}
follows also as a special case of Corollary~\ref{cor:dimintersections} below.

For later reference, let us define
\[
\alpha=\alpha(r_n)=\min\biggl\{t,\limsup_{n\to\infty}\frac{\log n}{-\log r_n}
  \biggr\}.
\]
Further, we set
\begin{equation}\label{nkdef}
\mathcal{N}_k=\{n\in\N\mid 2^{-(k+1)}\le r_n<2^{-k}\}\text{ and }
  n_k=\#\mathcal{N}_k,
\end{equation}
where the number of elements in a set $A$ is denoted by $\# A$. Finally, for
all analytic sets $F\subset X$, we define
\[
\mathcal{H}(F)=\{\x\in\Omega\mid E(\x)\cap F\neq\emptyset\}.
\]
In some of our results, we need to assume that there exists an increasing
sequence of positive integers $(k_i)_{i\in\N}$ such that
\begin{equation}\label{conditionC}
\lim_{i\to\infty}\frac{k_{i+1}}{k_i}=1\text{ and }
\lim_{i\to\infty}\frac{\log_2n_{k_i}}{k_i}=\alpha.
\end{equation}
This condition essentially means that, in a weak asymptotic sense, the
sequence $(r_n)_{n\in\N}$ behaves like $n^{-1/\alpha}$, see \cite{LiShiehXiao} for
various examples.

Now we are ready to state our first main theorem of this section.

\begin{theorem}\label{hittingcovering}
Let $F\subset X$ be an analytic set. With the above notation, we have
\begin{align}
\label{packingsmall}\P(\mathcal{H}(F))&=0\text{ if }\dimp F<t-\alpha,\\
\label{hausdorfflarge}\P(\mathcal{H}(F))&=1\text{ if }\dimh F>t-\alpha
  \text{ and}\\
\label{packinglarge}\P(\mathcal{H}(F))&=1\text{ if }\dimp F>t-\alpha
  \text{ and \eqref{conditionC} holds}.
\end{align}
\end{theorem}

In the circle $\T^1$, the analogues of \eqref{packingsmall} and
\eqref{packinglarge} were established by Li, Shieh and Xiao \cite{LiShiehXiao}.
Bugeaud and Durand \cite{BD} also recovered these results within the context of
Diophantine approximation. Li and Suomala \cite{LS} proved the analogue of
\eqref{hausdorfflarge} in the torus $\T^d$ and showed that the
assumption $\dimp F>d-\alpha$ alone is not enough to guarantee that
$\P(\mathcal{H}(F))>0$.

\begin{proof}[Proof of Theorem \ref{hittingcovering}]
We start from the claim \eqref{packingsmall}. For each $m\in\N$, let $Y_m$ be a
$(2^{-m})$-net of $X$ and $N_m=\#Y_m$. Since $X$ is compact and $t$-regular,
there is a constant $0<C_1<\infty$ such that
\begin{equation*}
C_1^{-1} 2^{tm}\le N_m\le C_1 2^{tm}
\end{equation*}
for all $m\in\N$. Using the fact that the packing dimension is equal
to the modified upper box counting dimension, which is due to 
Tricot \cite{Tricot82} (see also 
\cite[Proposition 3.8]{F03}), that is,
\[
\dimp F=\inf\biggl\{\sup_{n}\dimbu F_n\mid F\subset
\bigcup_{n=1}^\infty F_n\biggr\},
\]
it suffices to show that $\P(\{\x\in\Omega\mid E(\x)\cap F\neq\emptyset\})=0$
whenever $\dimbu F<t-\alpha$.

Let $\dimbu F<\gamma<t-\alpha$, $\alpha<\beta<t-\gamma$ and
\[
M_m=\#\{y\in Y_m\mid F\cap B(y,2^{-m})\neq\emptyset\}.
\]
Then $M_m<2^{m\gamma}$ and $n_m<2^{m\beta}$ for all $m$ large enough, say
$m\ge N_0$ (recall \eqref{nkdef}). Denote by $B_{m,1},\ldots, B_{m,M_m}$ the balls
among $\{B(y,2^{-m})\}_{y\in Y_m}$ which intersect $F$. For all $i=1,\dots,M_m$
and $n\in\mathcal{N}_m$, we have $B(x_n,r_n)\cap B_{m,i}\ne\emptyset$ only if
$d(x_n,y)\le 2^{-m+1}$, where $y$ is the centre of $B_{m,i}$. Since $x_n$ is
distributed according to $\mu$, \eqref{eq:d-regular} implies the existence of
$0<C_2<\infty$ such that
\[
\P(\{\x\in\Omega\mid B(x_n,r_n)\cap B_{m,i}\neq\emptyset\})\le C_2 2^{-mt}.
\]
Whence,
\begin{align*}
&\P\bigl(\{\x\in\Omega\mid B(x_n,r_n)\cap B_{m,i}\neq\emptyset\text{ for some }
  i\in\{1,\dots,M_m\}, n\in\mathcal{N}_m\}\bigr)\\
&\le C_2 n_m M_m 2^{-mt}\le C_2 2^{m(\gamma+\beta-t)}.
\end{align*}
Since $\gamma+\beta-t<0$, it follows from the Borel-Cantelli lemma that for
$\P$-almost all $\x\in\Omega$,
\[
\bigcup_{n\in\mathcal{N}_m}B(x_n,r_n)\cap F=\emptyset
\]
for all large enough $m$, implying that $E(\x)\cap F=\emptyset$.

The proof of \eqref{hausdorfflarge} given in \cite{LS} for the case $X=\T^d$
can be generalised in a straightforward way to the current setting (replace 
the dyadic cubes by the generalised dyadic cubes $\mathcal{Q}$ throughout). 
We will not repeat the
details.

To prove \eqref{packinglarge}, let $F\subset X$ with  $\dimp F>t-\alpha$.
Replacing $F$ by a subset if necessary, we may assume that $F$ is compact and
that $\dim_B (V\cap F)>t-\alpha$ whenever $V$ is an open set with
$V\cap F\neq\emptyset$ (see \cite{JP95}). It suffices to show that
\begin{equation}\label{infinitelymanyhits}
\P(\{\x\in\Omega\mid U(x_n,r_n)\cap V\cap F\neq\emptyset
 \text{ for infinitely many }n\in\N\})=1,
\end{equation}
where $U(x_n,r_n)$ is the interior of $B(x_n,r_n)$ and $V\subset X$ is an open
set such that $V\cap F\neq\emptyset$. Indeed, if this holds, letting $V$ run
over a countable base of the topology of $X$, it follows that for $\P$-almost
all $\x\in\Omega$, the set
\[
F\cap\bigcup_{n=k}^\infty U(x_n,r_n)
\]
is open and dense in $F$ for all $k\in\N$. Therefore, by the Baire's category
theorem,
\[
F\cap\limsup_{n\to\infty}B(x_n,r_n)\supset F\cap\limsup_{n\to\infty}U(x_n,r_n)
 \neq\emptyset.
\]

It remains to prove \eqref{infinitelymanyhits}. Fix $V\subset X$, let
$(k_i)_{i\in\N}$ be as in \eqref{conditionC} and define random variables
\[
S_i(\x)=\#\{n\in\mathcal{N}_{k_i}\mid U(x_n,r_n)\cap V\cap F\neq\emptyset\}
\]
for all $i\in\N$. Pick $\gamma$ and $\beta$ such that
$\dimbu (V\cap F)>\gamma>t-\beta>t-\alpha$. For all $m\in\N$, let $Z_m$ be a
$(2^{-m})$-net of $V\cap F$. Replacing $(k_i)_{i\in\N}$ by a subsequence, if
necessary, we may conclude by the first part of \eqref{conditionC} that
\begin{equation}\label{Mkiiso}
M_{k_i}:=\#Z_{k_i}>2^{\gamma k_i}\text{ for all }i\in\N.
\end{equation}
By \eqref{eq:d-regular}, there exists a constant $0<C_3<\infty$ such that, for
all $n\in\mathcal{N}_{k_i}$ and $z\in Z_{k_i}$, we have
\begin{equation}\label{fixednz}
\begin{split}
C_3^{-1}2^{-tk_i}\le\P(\{\x\in\Omega\mid z\in U(x_n,r_n)\})
   \text{ and}\\
\P(\{\x\in\Omega\mid z\in U(x_n,3r_n)\})\le C_3 2^{-tk_i}.
\end{split}
\end{equation}
Since $X$ is $t$-regular and $Z_{k_i}$ is a $(2^{-k_i})$-net, there is a constant
$0<C_4<\infty$ independent of $i\in\N$ such that any ball of radius $2^{-k_i}$
contains at most $C_4$ points $z\in Z_{k_i}$. Combining this with
\eqref{fixednz}, yields
\begin{equation}\label{fixedn}
\begin{split}
C_5^{-1}M_{k_i}2^{-tk_i}\le\P(\{\x\in\Omega\mid Z_{k_i}\cap U(x_n,r_n)\neq\emptyset
  \})\text{ and}\\
\P(\{\x\in\Omega\mid Z_{k_i}\cap U(x_n,3r_n)\neq\emptyset\})
  \le C_5 M_{k_i}2^{-tk_i},
\end{split}
\end{equation}
where $C_5=C_3C_4$. Note that $U(x_n,r_n)\cap V\cap F\neq\emptyset$ only if
$U(x_n,3r_n)\cap V\cap Z_{k_i}\neq\emptyset$.
Applying \eqref{fixedn} for all $n\in\mathcal{N}_{k_i}$, we conclude that
\begin{equation}\label{firstmoment}
C_5^{-1}n_{k_i}M_{k_i}2^{-tk_i}\le\E(S_i)\le C_5 n_{k_i}M_{k_i}2^{-tk_i}.
\end{equation}

To estimate the second moment $\E(S_i^2)$, note that the events
$Z_{k_i}\cap U(x_n,r_n)\neq\emptyset$ are independent for different $n\in\N$.
Thus, by \eqref{fixedn},
\begin{equation}\label{secondmoment}
\begin{split}
&\E(S_i^2)\le\sum_{p\in\mathcal{N}_{k_i}}\P(\{\x\in\Omega\mid Z_{k_i}\cap U(x_p,3r_p)
  \neq\emptyset\})\\
&+\sum_{\substack{p,q\in\mathcal{N}_{k_i}\\p\ne q}}\P(\{\x\in\Omega\mid Z_{k_i}\cap
  U(x_p,3r_p)\neq\emptyset\})\P(\{\x\in\Omega\mid Z_{k_i}\cap U(x_q,3r_q)
  \neq\emptyset\})\\
&\le C_5 n_{k_i} M_{k_i}2^{-tk_i}+(C_5)^2(n_{k_i})^2(M_{k_i})^22^{-2tk_i}.
\end{split}
\end{equation}
Since $\beta<\alpha$, we have that $n_{k_i}\ge 2^{k_i\beta}$ for all large
$i\in\N$ by \eqref{conditionC}. Recalling \eqref{Mkiiso} and $\gamma+\beta>t$,
we observe that the second term in the upper bound in \eqref{secondmoment} is
dominating. Whence, applying the Paley-Zygmund inequality, we conclude the
existence of a constant $0<C_6<\infty$ satisfying
\[
\P(\{\x\in\Omega\mid S_i(\x)>0\})\ge\frac{\E(S_i)^2}{\E(S_i^2)}\ge C_6>0
\]
for all $i\in\N$. Since the random variables $(S_i)_{i\in\N}$ are independent,
the Borel-Cantelli lemma implies that for $\P$-almost all $\x\in\Omega$,
$S_i(\x)>0$ for infinitely many $i\in\N$, completing the proof of
\eqref{infinitelymanyhits}.
\end{proof}

We will next investigate the dimension of the intersection $E(\x)\cap F$
aiming to generalise \cite[Theorem 2.4]{LiShiehXiao} to the metric setting. We
will need the following lemma, which is well known in the case $t\in\N$
and $X=[0,1]^t$. In the metric setting, we derive the result for a version of
fractal percolation using results of Lyons \cite{Lyons} on tree percolation.
Although we need this only when $X$ is $t$-regular, we note that the
proof works for any complete metric space satisfying a mild doubling condition.

\begin{lemma}\label{percolemma}
For any $s>0$, there is a probability space
$(\widetilde{\Omega},\Gamma,\widetilde{\P})$ and, for each
$\omega\in\widetilde{\Omega}$, a compact set $A(\omega)\subset X$ such that,
for any analytic set $E\subset X$, we have
\[
\widetilde{\P}(\{\omega\in\widetilde{\Omega}\mid A(\omega)\cap E
  =\emptyset\})=1
\]
provided $\dimh E<s$, while
\[
\|\dimh(A(\omega)\cap E)\|_{L^\infty(\widetilde{\P})}=\dimh E-s
\]
if $\dimh E\ge s$.
\end{lemma}

\begin{proof}
The model example of such random family of sets in the case $X=[0,1]^d$ is
given by a fractal percolation process and we will prove the lemma by
constructing an analogue of the fractal percolation in the space $X$ by using
the generalised dyadic cubes $\mathcal{Q}$ as defined in the beginning of this
section.

Let $p=2^{-s}$. For each $Q\in\mathcal{Q}$, we attach a Bernoulli random
variable $Z(Q)$ taking value $1$ with probability $p$ and value $0$ with
probability $1-p$. Further, we assume that these random variables are
independent for different $Q\in\mathcal{Q}$. We define the random fractal
percolation set as
\[
A=\bigcap_{n\in\N}\bigcup_{\substack{Q\in\mathcal{Q}_n\\Z(Q)=1}}\overline{Q},
\]
where $\overline Q$ is the closure of the set $Q$. Formally,
$\widetilde\Omega=\{0,1\}^{\mathcal Q}$, $\Gamma$ is the completion of the Borel
$\sigma$-algebra on $\widetilde\Omega$ and $\widetilde\P$ is the infinite
product over $\mathcal Q$ of the measures $(1-p)\delta_0+p\delta_1$.
There is an apparent tree structure behind the definition of $A$. Label each
$Q\in\mathcal{Q}$ with a vertex $v_Q$ and let $T$ be a graph with vertex set
$\{v_Q\}_{Q\in\mathcal Q}$. Draw an edge between vertices $v_{Q_{n,i}}$ and
$v_{Q_{m,j}}$ if and only if $|n-m|=1$ and $Q_{n,i}\cap Q_{m,i}\neq\emptyset$. Then
$T$ is a tree and we distinguish $v_X$ as its root.
The boundary of the tree $\partial T$ consists of all infinite paths
$v_0v_1v_2\ldots$, where $v_m=v_{Q_m}$ for some $Q\in\mathcal{Q}_m$ and
$Q_m\subset Q_n$, if $m\ge n$. Define a projection
$\Pi\colon\partial T\rightarrow X$ as
$\{\Pi(v_0v_1v_2\ldots)\}=\bigcap_{n=0}^\infty\overline{Q_n}$. Note that
$\Pi(\partial T)=X$. Then the law of $A$ given by
$\widetilde\P$ is the same as that of $\Pi(\widetilde{A})$, where
$\widetilde A\subset\partial T$ is determined by the component of the root in
the Bernoulli $p$-percolation on the tree $T$ (see \cite{Lyons}).

For $v=v_0v_1v_2\ldots, u=u_0u_1 u_2\ldots\in\partial T$, we define
\begin{equation*}
\kappa(v,u)=\begin{cases}
0 & \text{ if }v=u\\
2^{-\min\{i\,\mid\, v_i\neq u_i\}} & \text{ otherwise}.
\end{cases}
\end{equation*}
Then $(\partial T,\kappa)$ becomes a metric space. Using \eqref{inoutsideballs},
we find a constant $0<C<\infty$ such that $\diam(Q)\le C 2^{-n}$ for all
$n\in\N$ and $Q\in\mathcal{Q}_n$. Further, by \eqref{eq:d-regular}, there
exists a constant $0<\widetilde C<\infty$ such that for all
$2^{-n}\le r<2^{-n+1}$ and $x\in X$, the ball $B(x,r)$ may be covered by
$\widetilde C$ elements $Q\in\mathcal{Q}_n$. From these observations it readily
follows that if $E\subset X$ then
\begin{equation}
\dimh E=\dimhk\Pi^{-1}(E),
\end{equation}
where $\dimhk$ is the Hausdorff dimension with respect to the metric $\kappa$.
Whence, to finish the proof, we only need to verify that,
for any analytic set $F\subset\partial T$, we have
\begin{align*}
\widetilde{\P}(\{\omega\in\widetilde{\Omega}\mid\widetilde{A}\cap F
  =\emptyset\})&=1\text{ if }\dimh F<s\text{ and}\\
\|\dimh (A\cap F)\|_{L^\infty(\widetilde{P})}&=\dimh F-s\text{ if }\dimh F\ge s.
\end{align*}
But these results can be found in \cite[\S 7]{Lyons}, so we are done.
\end{proof}

We are now able to generalise the results of \cite[Theorem 1.4]{LiShiehXiao}
and \cite[Corollary 1.5]{LS} on the Hausdorff dimension of the intersections
$E(\x)\cap F$, when $F\subset X$ is a fixed analytic set.

\begin{theorem}\label{intersectionestimate}
If
$F\subset X$ is analytic, then for $\P$-almost every $\x\in\Omega$,
\[
\max\{0,\alpha+\dimh F-t\}\le\dimh (E(\x)\cap F)\le\max\{0,\alpha+\dimp F-t\}.
\]
\end{theorem}

\begin{proof}
With Lemma \ref{percolemma} and \eqref{hausdorfflarge} at hand, the lower
bound follows from a similar co-dimension argument as in \cite[Lemma 3.4]{KPX}.
The upper bound, in turn, can be obtained by a direct first-moment estimation
(see the proof of Theorem 1.4 in \cite{LiShiehXiao}). We omit the details.
\end{proof}

As an immediate consequence of Theorem \ref{intersectionestimate}, we obtain
the following corollary.

\begin{corollary}\label{cor:dimintersections}
Let $F\subset X$ be an analytic set
with $\dimh F=\dimp F$. Then for $\P$-almost all $\x\in\Omega$,
\[
\dimh (E(\x)\cap F)=\max\{0,\alpha+\dimh F-t\}.
\]
In particular, $\dimh E(\x)=\alpha$ for $\P$-almost all $\x\in\Omega$.
\end{corollary}

\begin{remark} \label{rem:BD}
A problem suggested by Mahler on how well can points, say, in the middle third
Cantor set $K$ be approximated by rational points, has stemmed a number of
results \cite{BFR11, BFKRW, EFS11, PV05, W01}, measuring the sizes of the
intersection of $K$ with sets $\mathcal W(\psi)$ from \eqref{eq:well} for
different values of the approximation speed $\psi$. These results are in part
inconclusive, and lead to conjectures on the size of the set
$K\cap \mathcal W(q^{-\tau})$ in \cite{BD, LSV07}. In particular, Bugeaud and 
Durand
\cite[Conjecture 1]{BD} conjectured that
\[
\dimh(K\cap\mathcal W(q^{-\tau}))=\max\{\dimh\mathcal W(q^{-\tau})+\dimh K-1,
  \tfrac 1\tau \dimh K\}.
\]
Following their reasoning, the first term in the maximum is realised for slow
approximation speeds, where the intervals giving the limsup set are big, and
the latter term appears when $\tau$ is large and the proportion of rationals
inside the set $K$ start to play a role. Bugeaud and Durand \cite{BD} have 
offered
evidence that supports their conjecture, in particular, by building a model of
random Diophantine approximation, where the centres of the generating intervals
are independent and uniformly distributed, either in the Cantor set or in the
whole circle $\mathbb T^1$ \cite[Section 2]{BD}. The random model is based on a
conjecture of Broderick, Fishman and Reich \cite{BFR11} on the proportion of
rationals in the Cantor set.

In our Corollary \ref{cor:dimintersections}, the size of the intersection of a
random covering set with a given analytic set $F$ is measured. 
This parallels the results of Bugeaud and Durand in the case of slow 
approximation speed, and could serve as a basis for more general results on 
random Diophantine approximation. We recall that, while our result concerns a 
wider class of sets than that of Bugeaud and Durand, the measure used by them 
to define the random variables is different from ours.
Notice that when the
generating sets are ball-like, many methods for the fast approximation also
apply directly (see \cite[Lemma 5]{BD}).
\end{remark}

\section{Covering sets in tori}\label{torus}

\subsection{Notations}

In this section, we study the hitting probabilities of random covering sets in
$d$-dimensional torus $\T^d$. We identify $\T^d$ with
$[-\tfrac 12,\tfrac 12[^d\subset\R^d$ and denote by $\Pi\colon\R^d\to\T^d$ the
natural covering map. We use the notation $\widetilde B$ for the lift of
$B\subset\T^d$ and denote the Lebesgue measure on $\T^d$ and $\R^d$ by $\L$.
We consider the probability space $(\Omega,\mathcal A,\P)$ which is the
completion of the infinite product of $(\T^d,\mathcal B(\T^d),\L)$. Let
$(A_n)_{n\in\N}$ be a sequence of analytic subsets of
$[-\tfrac 12,\tfrac 12[^d\subset\mathbb R^d$. For all $\x\in\Omega$, define the
covering set by
\[
E(\x)=\limsup_{n\to\infty}(x_n+\Pi(A_n)).
\]
It follows from Lemma~\ref{measurability} below that the set
$\{\x\in\Omega\mid\dimh E(\x)\le s\}$ is $\P$-measurable and, clearly, a tail
event for all $0\le s\le d$. Therefore, by Kolmogorov's zero-one law, there
exists $s_0\in[0,d]$ such that $\dimh E(\x)=s_0$ for $\P$-almost all
$\x\in\Omega$. The value of $s_0$ has been calculated for a general class of
sets $(A_n)_{n\in\N}$ in \cite[Theorem 1.1]{FJJS}.

\subsection{Upper bounds}

We start with a measurability lemma. Recall that a set is called analytic if
it is a continuous image of a Borel set, and a set is universally measurable if
it is $\mu$-measurable for any $\sigma$-finite Borel measure $\mu$.
By \cite[Theorem 21.10]{Ke}, analytic sets are universally measurable.

\begin{lemma}\label{measurability}
The covering set $E(\x)$ is analytic for all $\x\in\Omega$. Furthermore,
assuming that $F\subset\mathbb T^d$ is analytic and $t\in\mathbb R$, the sets
\begin{align*}
B&=\{(\x,z)\in\Omega\times\mathbb T^d\mid (E(\x)+z)\cap F=\emptyset\}
   \text{ and}\\
C&=\{(\x,z)\in\Omega\times\mathbb T^d\mid\dimh((E(\x)+z)\cap F)\le t\}
\end{align*}
are universally measurable.
\end{lemma}

\begin{proof}
The analyticity of $E(\x)$ follows from \cite[Proposition 14.4]{Ke}, which
states that the class of analytic sets is closed under countable unions and
intersections.

Define
\[
A=\{(\x,z,y)\in\Omega\times\T^d\times\T^d\mid y\in E(\x)+z\}.
\]
To prove that $A$ is analytic, it is enough to show that
\[
D_n=\{(\x,z,y)\in\Omega\times\T^d\times\T^d\mid y\in x_n+ A_n+z\}
\]
is analytic for all $n\in\N$. Define
$f_n\colon\Omega\times\T^d\times\T^d\to\T^d$ by $f_n(\x,z,y)=y-x_n-z$. Note that
$D_n=f_n^{-1}(A_n)$. Since $f_n$ is continuous and $A_n$ is analytic, $D_n$ is
analytic by \cite[Proposition 14.4]{Ke}. Let
$\pi_{12}\colon\Omega\times\T^d\times\T^d\to\Omega\times\T^d$
be the projection $\pi_{12}(\x,z,y)=(\x,z)$. The observation that the complement
of $B$ is $B^c=\pi_{12}\bigl(A\cap(\Omega\times\T^d\times F)\bigr)$
implies that $B^c$ is analytic and, thus, $B$ belongs to the $\sigma$-algebra
generated by analytic sets. Therefore, $B$ is universally measurable by
\cite[Theorem 21.10]{Ke}.

In \cite{Del}, it is shown that if $S$ and $X$ are compact metric spaces and
$G\subset S\times X$ is analytic, the map $H(s)=\dimh(G\cap(\{s\}\times X)$ is
measurable with respect to the $\sigma$-algebra generated by analytic sets.
Letting $S=\Omega\times\T^d$, $X=\T^d$ and $G=A\cap(S\times F)$, we have that
$C=H^{-1}([0,t])$. Therefore, the analyticity of $A$ and $F$ implies that $C$
is universally measurable.
\end{proof}

The following theorem is a counterpart of the upper bound parts of
Theorems \ref{hittingcovering} and \ref{intersectionestimate}.

\begin{theorem}\label{upper}
Let $F\subset\mathbb T^d$ be analytic. Then for $\P$-almost all $\x\in\Omega$,
\begin{align*}
&E(\x)\cap F=\emptyset\text{ if }\dimp F<d-s_0\text{ and}\\
&\dimh(E(\x)\cap F)\le s_0+\dimp F-d\text{ if }\dimp F\ge d-s_0,
\end{align*}
where $s_0$ is the $\P$-almost sure Hausdorff dimension of $E(\x)$.
\end{theorem}

\begin{proof}
 It is known that (cf. \cite{Tricot82} or \cite[Theorem 8.10]{M95})
\begin{equation}\label{eq12}
\dimh (E(\x)\times F)\le\dimh E(\x)+\dimp F
\end{equation}
for all $\x\in\Omega$. In \cite[Theorem 8.10]{M95}, the result is stated only
for Borel sets but the part of the proof where inequality \eqref{eq12} is
proven is valid for all sets. Lifting $E(\x)$ and $F$ to $\R^d$ as
$\widetilde E(\x$) and $\widetilde F$, we may apply
\cite[(13.2)]{M95}, which implies that for all $\x\in\Omega$ and $z\in\R^d$,
\begin{equation}\label{slice}
(\widetilde E(\x)+z)\cap\widetilde F=\pi_1((\widetilde E(\x)\times\widetilde F)
  \cap V_z),
\end{equation}
where $V_z=\{(u,v)\in\mathbb R^d\times\mathbb R^d\mid u-v=z\}$ and
$\pi_1(u,v)=u$.

We study first the case $\dimp F<d-s_0$. From \eqref{eq12} we deduce that the
inequality $\dimh(E(\x)\times F)<d$ holds for $\P$-almost all
$\x\in\Omega$. Therefore, since the dimension will not increase under the
projection onto the orthogonal complement of $V_0$ (see
\cite[Theorem 7.5]{M95}),
$(\widetilde E(\x)\times\widetilde F)\cap V_z=\emptyset$ for $\L$-almost all
$z\in\R^d$. Projecting the sets back to $\T^d$, we have by \eqref{slice} that,
for $\P$-almost all $\x\in\Omega$, $(E(\x)+z)\cap F=\emptyset$ for $\L$-almost
all $z\in\T^d$. By virtue of Lemma~\ref{measurability}, the set
\[
B=\{(\x,z)\in\Omega\times\T^d\mid (E(\x)+z)\cap F=\emptyset\}
\]
is universally measurable. Thus, Fubini's theorem
implies that for $\L$-almost all $z\in\T^d$, $(E(\x)+z)\cap F=\emptyset$
for $\P$-almost all $\x\in\Omega$. In particular, there exists $z_0\in\T^d$ with
\[
\P(\{\x\in\Omega\mid (E(\x)+z_0)\cap F=\emptyset\})=1.
\]
Set $\z=(z_i)_{i\in\N}$, where $z_i=z_0$ for all $i\in\N$. Since
$E(\x)+z_0=E(\x+\z)$ and $\P$ is translation invariant, we have
\[
\P(\{\x\in\Omega\mid E(\x)\cap F=\emptyset\})=1.
\]

Now we consider the case $\dimp F\ge d-s_0$. As at the beginning
of this section, observe that
$\{\x\in\Omega\mid\dimh(E(\x)\times F)\le s\}$ is a tail event
for all $0\le s\le 2d$ and, by Kolmogorov's zero-one law, there exists
$t_0\in[0,2d]$ such that $\dimh(E(\x)\times F)=t_0$ for $\P$-almost all
$\x\in\Omega$. If $t_0<d$, the above argument implies that
$E(\x)\cap F=\emptyset$ for $\P$-almost all $\x\in\Omega$ and, thus, the second
claim in the statement of theorem is true. If $t_0\ge d$, we may apply
\cite[Theorem 13.12]{M95} which implies that, for $\P$-almost all
$\x\in\Omega$,
\[
\dimh\bigl((E(\x)+z)\cap F\bigr)\le\dimh(E(\x)\times F)-d
\]
for $\L$-almost all $z\in\T^d$. Observe that \cite[Theorem 13.12]{M95}
is stated only for Borel sets but that assumption is not used in the proof.
Furthermore, by \eqref{eq12}, Lemma~\ref{measurability} and Fubini's theorem,
for $\L$-almost all $z\in\T^d$,
\[
\dimh\bigl((E(\x)+z)\cap F\bigr)\le s_0+\dimp F-d
\]
for $\P$-almost all $\x\in\Omega$. Now the claim follows as above from the
translation invariance of $\P$.
\end{proof}

\subsection{A counter example for non-trivial lower bounds}\label{parallel}

In this subsection, we give an application of our results to a specific class
of affine random covering sets in $\T^d$ which motivated this work. This class
also demonstrates why the analogues of the lower bounds given by
Theorems \ref{hittingcovering} and \ref{intersectionestimate} are not true in
the setting of general generating sets $A_n$, $n\in\N$.

Let $(r_n)_{n\in\N}$ be a decreasing sequence of numbers between 0
and 1 tending to zero. Further, let
$0<H_d\le H_{d-1}\le\cdots\le H_2\le H_1=1$. We consider the covering set
\[
E(\x)=\limsup_{n\to\infty}(x_n+\Pi(A_n))
\]
for $\x\in\Omega$, where
\begin{equation}\label{eq:sidelengths}
A_n=\prod_{i=1}^d[-\tfrac 12,-\tfrac 12+(r_n)^{(H_i)^{-1}}]\subset\R^d
\end{equation}
are rectangles in $\R^d$ with sides parallel to the coordinate axes and
side lengths given by \eqref{eq:sidelengths}.

Let $s\in [0,d]$. We denote the integer and fractional parts of
$s$ by $\lfloor s\rfloor$ and $\{s\}$, respectively. For all $n\in\N$, we
have
\[
\Phi^s(A_n)=r_n^{\sum_{i=1}^{\lfloor s\rfloor}(H_i)^{-1}+\{s\}
  (H_{\lfloor s\rfloor+1})^{-1}},
\]
where $\Phi^s$ is the singular value function of a rectangle determined by its
side lengths (see \cite{JJKLS}). Let
$k_0=\max\{k\in\{1,\dots,d\}\mid\sum_{i=1}^k(H_i)^{-1}\le\alpha\}$ and
$\alpha=\min\{d,\limsup_{n\to\infty}\frac{\log n}{-\log r_n}\}$.
By \cite[Theorem 2.1]{JJKLS},
\[
\dimh E(\x)=s_0=\min\Bigl\{d,\inf\bigl\{s\ge 0\mid\sum_{n=1}^\infty\Phi^s(A_n)
  <\infty\bigr\}\Bigr\}.
\]
Combining this with the second equality in Theorem~\ref{dimcoveringset}, we
conclude that $\lfloor s_0\rfloor=k_0$ and
$\{s_0\}=H_{k_0+1}(\alpha-\sum_{i=1}^{k_0}(H_i)^{-1})$. Therefore, for $\P$-almost
all $\x\in\Omega$,
\[
\dimh E(\x)=s_0=\min\bigl\{d,\alpha H_{k_0+1}
 +\sum_{i=1}^{k_0}\bigl(1-\frac{H_{k_0+1}}{H_i}\bigr)\bigr\}.
\]
Thus, from Theorem \ref{upper}, we conclude that almost surely,
\begin{equation}\label{yy}
\begin{split}
&E(\x)\cap F=\emptyset\text{ if }\dimp F<d-s_0\text{ and}\\
&\dimh(E(\x)\cap F)\le s_0+\dimp F-d\text{ if }\dimp F\ge d-s_0.
\end{split}
\end{equation}

The following example shows that we cannot have similar lower
bounds in Theorem~\ref{upper} as in Theorems~\ref{hittingcovering} and
\ref{intersectionestimate}.

\begin{example}\label{bad case}
Fix $0<\varepsilon<1$. Consider $A_n\subset\R^2$, $n\in\N$, as in
\eqref{eq:sidelengths}, where $H_2=\frac\varepsilon{1+\varepsilon}$ and
$r_n=n^{-\varepsilon}$ for all $n\in\N$. Then $\alpha=\varepsilon^{-1}$ and, thus,
the $\P$-almost sure value of the Hausdorff dimension of the covering set is
$s_0=1+\frac{1-\varepsilon}{1+\varepsilon}$. If $-\frac 12\le b<\frac 12$ and
$F=\Pi([-\frac 12,\frac 12]\times \{b\})$, then
$\dimh F=\dimp F=1>2-s_0=\frac{2\varepsilon}{1+\varepsilon}$. However,
\[
\mathbb{P}(\{\x\in\Omega\mid E(\x)\cap F\neq\emptyset\})=0.
\]
Indeed, $E(\x)\cap F\neq\emptyset$ only if
$\Pi(0,b)\in\limsup_{n\to\infty} \Pi(\pi_2(\widetilde{x_n}+A_n))$, where
$\pi_2$ denotes the orthogonal projection onto the $y$-axis. Now
the law of the set $\limsup_{n\to\infty} \Pi(\pi_2(\widetilde{x_n}+A_n))$
is that of a random covering set in $\T^1$ with
$r_n=n^{-\varepsilon(H_2)^{-1}}=n^{-1-\varepsilon}$. Since
$\sum_{n=1}^\infty n^{-1-\varepsilon}<\infty$, the Borel-Cantelli lemma implies
\[
\P\bigl(\Pi(0,b)\in\limsup_{n\to\infty} \Pi(\pi_2(\widetilde{x_n}+A_n))\bigr)=0.
\]
\end{example}

We will next show how the results from Section \ref{sec:hitting_metric} may be
used to replace \eqref{yy} by a sharper estimate and also to get an analogy of
the lower bounds. To that end, we define a new metric $\kappa$ on $\T^d$ by
`snowflaking' the Euclidean distance by factor $H_i$ in each coordinate
direction. More precisely, for all $y,z\in\T^d$, we set
\[
\kappa(z,y)=\max_{1\le i\le d}2^{H_i}|z_i-y_i|^{H_i},
\]
where the natural distance between points $a,b\in\T^1$ is denoted
by $|a-b|$. With this metric, $\T^d$ becomes a $t$-regular metric space with
$t=\sum_{i=1}^d (H_i)^{-1}$ and $\mathcal{L}$ is the $t$-regular measure
satisfying \eqref{eq:d-regular}. Further, in this metric each $\Pi(A_n)$ is a
ball of radius $r_n$. The constants $2^{H_i}$ appear in the definition of
$\kappa$ to ensure this. Thus, we are in a situation where the results from
Section~\ref{sec:hitting_metric} may be applied. For instance, for $\P$-almost
all $\x\in\Omega$, we have
\[
\dimhk E(\x)=\alpha=\min\Bigl\{t,\limsup_{n\to\infty}
\frac{\log n}{-\log r_n}\Bigr\},
\]
where the Hausdorff and packing dimensions with respect to the metric $\kappa$
are denoted by $\dimhk$ and $\dimpk$, respectively. Further, if $F\subset\T^d$
is analytic, then for $\P$-almost all $\x\in\Omega$,
\begin{equation}\label{stronger}
\begin{split}
&E(\x)\cap F=\emptyset\text{ if }\dimpk F<t-\alpha,\\
&E(\x)\cap F\neq\emptyset\text{ if }\dimhk F>t-\alpha,\\
&E(\x)\cap F\neq\emptyset\text{ if }\dimpk F>t-\alpha
   \text{ and \eqref{conditionC} holds, and}\\
&\alpha+\dimhk F-t\le \dimhk(E(\x)\cap F)\le\alpha+\dimpk F-t\text{ if }
   \dimpk F>t-\alpha.
\end{split}
\end{equation}

\begin{remark}
a) In Example \ref{bad case}, we have $\dimpk F=\dimhk F=1$,
$t=(1+2\varepsilon)\varepsilon^{-1}$ and $\alpha=\varepsilon^{-1}$, implying
$1=\dimpk F<t-\alpha=2$ (which is consistent with \eqref{stronger}).

b) Example~\ref{bad case} shows that there cannot be any non-trivial lower
bound in \eqref{yy} depending only on the Hausdorff and packing dimensions of
$E(\x)$ and $F$. However, \eqref{stronger} demonstrates that there can be some
other quantities like $\dimhk$ and $\dimpk$ which do imply non-trivial lower
bounds.
\end{remark}

\subsection{Covering sets involving random rotations}

As demonstrated in Example \ref{bad case}, the analogues of
\eqref{hausdorfflarge} and \eqref{packinglarge}, and the lower bound in
Theorem~\ref{intersectionestimate} are not always true for a sequence of
general sets $(A_n)_{n\in\N}$. In Example \ref{bad case}, this is due to the
alignment of the rectangles $A_n$, $n\in\N$. In this subsection, we consider a
slightly modified and, perhaps, more natural version of random covering sets in
$\T^d$, where the sets $A_n$ are also rotated, and show that under a standard
additional assumption on the dimensions, the analogues of
\eqref{hausdorfflarge} and Theorem~\ref{intersectionestimate} remain valid for
any sequence of analytic generating sets $A_n$.

Let $(\widetilde\Omega,\widetilde{\mathcal A},\widetilde\P)$ be the completion
of the infinite product of
$(\T^d\times\mathcal O(d),\mathcal B(\T^d\times\mathcal O(d)),\L\times\theta)$,
where $\mathcal O(d)$ is the orthogonal group on $\R^d$ and $\theta$ is
the Haar measure on $\mathcal O(d)$. Assume that
$A_n\subset U(0,\frac 12)\subset\R^d$ are analytic for all $n\in\N$, where
$U(x,r)$ is the open ball with radius $r$ centred at $x$. We
consider the covering sets
\begin{equation*}
E(\x,\h)=\limsup_{n\to\infty}\bigl(x_n+\Pi(h_n(A_n))\bigr)\subset\T^d\,,
\end{equation*}
for $(\x,\h)\in\widetilde{\Omega}$.

\begin{remark}\label{rotation}
a) Since $A_n\subset U(0,\frac 12)$, the restriction of $\Pi$ to $h_n(A_n)$ is
injective for all $h_n\in\mathcal O(d)$ and $n\in\N$.

b) As above, Kolmogorov's zero-one law implies the existence of $s_0^R\in[0,d]$
such that $\dimh E(\x,\h)=s_0^R$ for $\widetilde\P$-almost all
$(\x,\h)\in\widetilde\Omega$. Note that the value $s_0$ for the Hausdorff
dimension of typical random covering sets calculated in \cite{FJJS} depends only
on the shapes of the generating sets $A_n$, that is, the value of $s_0$ for the
generating sequence $(A_n)_{n\in\N}$ is equal to that of the sequence
$(h_n(A_n))_{n\in\N}$ for all $\h\in(\mathcal O(d))^\N$.
Thus, for this large class of sets, we have $s_0^R=s_0$, that is, adding the
rotations will not change the dimension of typical random covering sets.

c) Notice that Theorem \ref{upper} holds for $E(\x,\h)$ too, as the proof
is valid for any fixed sequence of rotations $(h_n)_{n\in\N}$.
\end{remark}

Before proving our next main theorem, which gives a lower 
bound for typical intersections, we
prove two lemmas. The first one states that the dimension of a typical
covering set is the same inside every ball.

\begin{lemma}\label{invariance}
For $\widetilde\P$-almost all $(\x,\h)\in\widetilde\Omega$,
\[
\dimh\bigl(E(\x,\h)\cap B(z,r)\bigr)=s_0^R
\]
for all $z\in\T^d$ and $r>0$.
\end{lemma}

\begin{proof}
For all $z\in\T^d$, let $\z\in(\T^d)^\N$ be the sequence such that
$z_n=z$ for all $n\in\N$. Since $\L$ is translation invariant, we have for
all $t\in[0,d]$ and $z\in\T^d$ that
\begin{equation}\label{transinv}
\begin{split}
&\widetilde\P\bigl(\{(\x,\h)\in\widetilde\Omega\mid\dimh\bigl(E(\x,\h)\cap
  B(z,r)\bigr)=t\}\bigr)\\
&=\widetilde\P\bigl(\{(\x,\h)\in\widetilde\Omega\mid\dimh\bigl(E(\x+\z,\h)
  \cap B(z,r)\bigr)=t\}\bigr)\\
&=\widetilde\P\bigl(\{(\x,\h)\in\widetilde\Omega\mid\dimh\bigl(E(\x,\h)\cap
  B(0,r)\bigr)=t\}\bigr).
\end{split}
\end{equation}
As above, we see that, for every $r>0$, there exists a unique $t_r\in [0,d]$
such that
$\P\bigl(\bigl\{\x\in\Omega\mid\dimh\bigl(E(\x,\h)\cap B(z,r)\bigr)
  =t_r\bigr\}\bigl)=1$.
Since $\T^d$ may be covered by a finite number of balls with radius $r$,
\eqref{transinv} implies that $t_r=s_0^R$ for all $r>0$.
\end{proof}

Unlike the translation, the rotation is not well defined on $\T^d$. To deal
with the technical problems caused by this fact, we need the following lemma.

\begin{lemma}\label{reducesmall}
Let $0<r<\frac 12$. There exist $M\in\N$ and $z_i\in U(0,\frac 12)\subset\R^d$,
$i=1,\dots,M$, such that, for all $n\in\N$, we may decompose
$A_n=\bigcup_{i=1}^M A_n^i$ into Borel sets $A_n^i$ in such a way that
$A_n^i\subset B(z_i,r)$ for all
$n\in\N$ and, for all $(\x,\h)\in\widetilde\Omega$, there exists
$i\in\{1,\dots,M\}$ such that
\[
\dimh E(\x,\h)=\dimh\Bigl(\limsup_{n\to\infty}\bigl(x_n+\Pi(h_n(A_n^i))\bigr)
  \Bigr).
\]
Further, for all $i\in\{1,\dots,M\}$, there exist sets
$\widehat A_n^i\subset B(0,r)$, $n\in\N$, and a bijection
$F_i\colon\widetilde\Omega\to\widetilde\Omega$ preserving $\widetilde\P$,
that is, $(F_i)_*\widetilde\P=\widetilde\P$, such that
\[
\limsup_{n\to\infty}\bigl(x_n+\Pi(h_n(A_n^i))\bigr)=\limsup_{n\to\infty}\Bigl(\bigl(
  \pi_1(F_i(\x,\h))\bigr)_n+\Pi\bigl(\bigl(\pi_2(F_i(\x,\h))\bigr)_n
  (\widehat A_n^i)\bigr)\Bigr)
\]
for all $(\x,\h)\in\widetilde\Omega$, where $\pi_1(\x,\h)=\x$ and
$\pi_2(\x,\h)=\h$.
\end{lemma}

\begin{proof}
The desired decomposition may be obtained by any partitioning of
$U(0,\frac 12)$ into Borel subsets of diameter less than $r$ and for any choice
of points $z_i$ inside these sets since
\[
\limsup_{n\to\infty}(C_n\cup D_n)=\limsup_{n\to\infty}C_n\cup\limsup_{n\to\infty}D_n
\]
for all sets $C_n,D_n\subset\T^d$, $n\in\N$. For the second claim, define
$\widehat A_n^i=A_n^i-z_i\subset B(0,r)$ and
$F_i(\x,\h)=\bigl(\x+\bigl(\Pi(h_n(z_i))\bigr)_{n\in\N},\h\bigr)$ for all
$i\in\{1,\dots,M\}$. The translation invariance of $\L$ yields
$(F_i)_*\widetilde\P=\widetilde\P$. Further, since $z_i\in U(0,\frac 12)$ and
$A_n^i\subset U(0,\frac 12)$ for all $n\in\N$ and
$i\in\{1,\dots,M\}$,
\[
x_n+\Pi(h_n(A_n^i))=x_n+\Pi(h_n(z_i))+\Pi(h_n(\widehat A_n^i))
\]
for all $x_n\in\T^d$, implying the last claim.
\end{proof}

Now we are ready to prove the second main theorem of this section.

\begin{theorem}\label{hittingresult_dimh}
Let $F\subset\T^d$ be an analytic set with $\dimh F>d-s_0^R$. Assume
that $\max\{s_0^R,\dimh F\}>\frac 12(d+1)$. Then
\[
\dimh(E(\x,\h)\cap F)\ge s_0^R+\dimh F-d
\]
for $\widetilde\P$-almost all $(\x,\h)\in\widetilde\Omega$.
\end{theorem}

\begin{proof}
Since $E(\x,\h)\cap(F-z)=-z+E(\x+\z,\h)\cap F$ for all
$(\x,\h)\in\widetilde\Omega$ and $z\in\T^d$ (where $\z=(z,z,\ldots)$) and
$\widetilde\P$ is translation invariant, we may assume that
$\dimh(\Pi(V)\cap F)=\dimh F$, where $V=B(0,\frac 1{10})\subset\R^d$. Fix
$t<s_0^R+\dimh F-d$. By Lemma~\ref{invariance},
$\dimh(E(\x,\h)\cap\Pi(V))=s_0^R$ for $\widetilde\P$-almost all
$(\x,\h)\in\widetilde\Omega$. According to a general intersection result for
Hausdorff dimension \cite[Theorem 13.11]{M95}, if $A,B\subset\R^d$ are analytic
sets with $\dimh A+\dimh B>d$ and $\max\{\dimh A,\dimh B\}>\tfrac 12(d+1)$
then, for $\theta$-almost all $h\in\mathcal{O}(d)$,
\[
\L\bigl(\bigl\{z\in\R^d\mid\dimh (h(A)+z)\cap B)
  >\dimh A-\dimh B-d-\varepsilon\bigr\}\bigr)>0
\]
for any $\varepsilon>0$. Note that \cite[Theorem 13.11]{M95} is stated for
Borel sets but the proof given is valid for analytic sets as well since
Frostman's lemma is valid for analytic sets \cite{Ca}. In
\cite[Theorem 13.11]{M95}, the theorem
is stated in asymmetric way, but the above symmetric form is also valid by the
translation invariance of $\L$ and rotation invariance of $\theta$. Thus, for
every realisation of $E$ with $\dimh(E(\x,\h)\cap\Pi(V))=s_0^R$, we have, for
$\theta$-almost all $h\in\mathcal O(d)$, that
\begin{equation}\label{Pertti}
\L\bigl(\bigl\{z\in B(0,\tfrac 15)\mid\dimh\bigl((h(\widetilde E(\x,\h))+z)
  \cap\widetilde F\cap V\bigr)\ge t\bigr\}\bigr)>0.
\end{equation}
By Lemma~\ref{measurability} and
Fubini's theorem, for all $\varepsilon>0$, there exist
$h_0\in B(\Id,\varepsilon)\subset\mathcal O(d)$ and
$z_0\in B(0,\frac 15)\subset\R^d$ such that
\[
\widetilde\P\bigl(\bigl\{(\x,\h)\in\widetilde\Omega\mid\dimh\bigl(
 (h_0(\widetilde E(\x,\h))+z_0)\cap\widetilde F\cap V\bigr)\ge t\bigr\}\bigr)>0.
\]
Let $\w\in(\T^d)^\N$, where $w_n=\Pi(h_0^{-1}(z_0))$ for all $n\in\N$. Using the
fact that $z_0\in B(0,\frac 15)$, we may choose small enough $\varepsilon>0$
depending on $d$ only such that
\[
(h_0(\widetilde E(\x,\h))+z_0)\cap V=h_0\bigl(\widetilde E(\x,\h)+h_0^{-1}(z_0)
  \bigr)\cap V=h_0(\widetilde E(\x+\w,\h))\cap V.
\]
Thus, the translation invariance of $\P$ implies
\begin{equation}\label{Ppos}
\widetilde\P\bigl(\bigl\{(\x,\h)\in\widetilde\Omega\mid\dimh\bigl(
  h_0(\widetilde E(\x,\h))\cap\widetilde F\cap V\bigr)\ge t\bigr\}\bigr)>0.
\end{equation}

For $z\in\T^d$, let $\tilde z\in [-\frac 12,\frac 12[^d$ be the unique element
such that $\Pi(\tilde z)=z$. Define $\hat h_0\colon\T^d\to\T^d$ by
\[
\hat h_0(z)=\begin{cases}
        \Pi(h_0(\tilde z)),&\text{ if }\tilde z\in B(0,\frac 15)\\
        z,&\text{ if }\tilde z\in [-\frac 12,\frac 12[^d\setminus B(0,\frac 15).
      \end{cases}
\]
Then $(\hat h_0)_*\L=\L$ and $(H_0)_*\widetilde\P=\widetilde\P$, where
$H_0(\x,\h)=\bigl((\hat h_0(x_n))_{n\in\N},\h\bigr)$. By Lemma~\ref{reducesmall},
we may assume that $A_n\subset V$ for all $n\in\N$. Therefore, decreasing
$\varepsilon$ if necessary, we have
\begin{equation}\label{rotinvariance}
\Pi\bigl(h_0(\tilde x_n+h_n(A_n))\bigr)\cap\Pi(V)
  =\Bigl(\hat h_0(x_n)+\Pi\bigl(h_0(h_n(A_n))\bigr)\Bigr)\cap\Pi(V)
\end{equation}
for all $n\in\N$ and $(\x,\h)\in\widetilde\Omega$. Using
$(H_0)_*\widetilde\P=\widetilde\P$ and combining \eqref{Ppos} with
\eqref{rotinvariance}, we conclude
\[
\widetilde\P\bigl(\bigl\{(\x,\h)\in\widetilde\Omega\mid\dimh\bigl(\widetilde E
 (\x,(h_0h_n)_{n\in\N})\cap\widetilde F\cap V\bigr)\ge t\bigr\}\bigr)>0.
\]
Since $\theta$ is the Haar measure, $(h_0)_*\theta=\theta$ and, thus,
\[
\widetilde\P\bigl(\bigl\{(\x,\h)\in\widetilde\Omega\mid
  \dimh(E(\x,\h)\cap F)\ge t\bigr\}\bigr)>0.
\]
Since $\{(\x,\h)\in\widetilde\Omega\mid\dimh(E(\x,\h)\cap F)\ge t\}$ is a tail
event, $\dimh(E(\x,\h)\cap F)\ge t$ for $\widetilde\P$-almost all
$(\x,\h)\in\widetilde\Omega$ by Kolmogorov's zero-one law. The proof is
completed by letting $t$ tend to $s_0^R+\dimh F-d$ along a sequence.
\end{proof}

Combining Theorems~\ref{upper} and \ref{hittingresult_dimh} with
Remark~\ref{rotation}.c), gives the following corollary.

\begin{corollary}\label{cor}
Let $F\subset\T^d$ be an analytic set with $\dimh F>d-s_0^R$. Assume that
$\max\{s_0^R,\dimh F\}>\frac 12(d+1)$. Then, for $\widetilde\P$-almost all
$(\x,\h)\in\widetilde\Omega$,
\begin{equation}\label{eq11}
s_0^R+\dimh F-d\le\dimh(E(\x,\h)\cap F)\leq s_0^R+\dimp F-d.
\end{equation}
In particular, if $\dimh F=\dimp F$, then, for $\widetilde\P$-almost all
$(\x,\h)\in\widetilde\Omega$,
\[
\dimh(E(\x,\h)\cap F)=s_0^R+\dimh F-d.
\]
\end{corollary}

\begin{remark}\label{conclution}
a) Observe that $\dimh F>d-s_0^R$ implies $\max\{s_0^R,\dimh F\}>\frac d2$. We
do not know whether the assumption $\max\{s_0^R,\dimh F\}>\frac 12(d+1)$ is
needed in Theorem~\ref{hittingresult_dimh}. This is a famous open problem in
the theory of intersections of general sets.  In \cite{LS}, Li and Suomala
constructed examples featuring that the lower and upper bounds in \eqref{eq11}
may be achieved. Thus one cannot find better bounds than those in
\eqref{eq11} involving only the Hausdorff and packing dimensions of $E(\x,\h)$
and $F$. However, in \cite{LS} there is an example where $\dimh(E(\x,\h)\cap F)$
is almost surely strictly between the bounds given in \eqref{eq11}.

b) The problem of exceptional geometry of a limsup set produced from
axes-parallel rectangles such as in Example \ref{bad case} seems to be a
prevalent phenomenon. As examples, see \cite[Section 6]{WWX15} and
\cite[Example 9.10]{F03}. Theorem \ref{hittingresult_dimh} offers further
evidence to support the folklore conjecture that this is caused by the atypical
exact alignment of the construction sets, and can be overcome by re-orienting
them which, in our case, was done by random rotations.

c) What Corollary~\ref{cor} implies in the framework of Bugeaud and Durand
\cite{BD} (recall Remark \ref{rem:BD}) is that, under the additional dimension
assumptions, the random Diophantine approximation properties for points in a
fixed analytic set $F$ are valid for any choice of shapes of the generating
sets, as long as the sets are randomly rotated. Notice that, for a large
selection of generating sets, the dimension does not vary with the rotations
\cite{FJJS}.
\end{remark}

\end{document}